\documentclass[11pt,a4paper]{article}%
\usepackage{amssymb}
\usepackage{amsfonts}
\usepackage{amsmath}
\usepackage[left=3cm,right=3cm,top=3cm,bottom=3.5cm]{geometry}
\usepackage{graphicx}
\usepackage{pstricks}
\usepackage{pst-text}
\usepackage{enumitem}
\usepackage{pdflscape}
\usepackage{multirow}
\usepackage[latin1]{inputenc}
\usepackage{amsthm}
\providecommand{\U}[1]{\protect\rule{.1in}{.1in}}

\providecommand{\R}{\mathbb{R}}

\newtheorem{theorem}{Theorem}[section]

\newtheorem{defi}[theorem]{Definition}

\newtheorem{lemma}[theorem]{Lemma}

\newenvironment{remark}{\noindent\textbf{Remark.}\upshape}{}

\numberwithin{equation}{section}
\providecommand{\N}{\mathbb{N}}

\begin{document}

\title{Monotonicity of the Lebesgue constant for equally spaced knots}
\author{Markus Passenbrunner}
\maketitle

\begin{abstract}
Let $t_{i}=\frac{i}{n}$ for $i=0,\ldots ,n$ be equally spaces knots in
the unit interval $[0,1].$ Let $\mathcal{S}_{n}$ be the space of piecewise
linear continuous functions on $[0,1]$ with knots $\pi _{n}=\{t_{i}:0\leq
i\leq n\}.$ Then we have the orthogonal projection $P_{n}$ of $L^{2}([0,1])$
onto $\mathcal{S}_{n}.$ In Section 1 we collect a few preliminary facts about the solutions of the recurrence $f_{k-1}-4f_{k}+f_{k+1}=0$ that we need in Section 2 to show that the sequence $%
a_{n}=\left\Vert P_{n}\right\Vert _{1}$ of $L^{1}-$norms of these projection
operators is strictly increasing. 
\end{abstract}

\section{Solutions of $f_{k-1}-4f_{k}+f_{k+1}=0$ and their Properties}
This section is to define and examine a few properties of the solutions of the recurrence $f_{k-1}-4f_{k}+f_{k+1}=0$, which we will use extensively in the sequel.
For an arbitrary real number $x$, let $A_{x}:=\cosh(\alpha x)$ and $\sqrt{3}B_{x}:=\sinh(\alpha x)$ with
$\alpha>0$ defined by $\cosh\alpha=2.$ For $k\in\N_0$, $A_{k}$ and $B_{k}$ can also be defined
by the recurrence relations%
\begin{eqnarray}
A_{k+1}&=& 2A_{k}+3B_{k}\quad\text{with }A_{0}=1,\label{eq:recA}\\
B_{k+1}&=& A_{k}+2B_{k}\quad\text{with }B_{0}=0.\label{eq:recB}
\end{eqnarray}
This follows from the basic identities
\begin{eqnarray}
\cosh(x+y)  & =&\cosh x\cosh y+\sinh x\sinh y,\label{eq:hyp1}\\
\sinh(x+y)  & =&\sinh x\cosh y+\cosh x\sinh y.\label{eq:hyp2}%
\end{eqnarray}
The crucial fact about $A_k$ and $B_k$ is that they are independent solutions of the linear recursion $f_{k-1}-4f_{k}+f_{k+1}=0$ and this recursion in turn takes into account the special form of the Gram matrix for equally spaces knots (see (\ref{eq:invgram})).
We note that it is easy to see 
that the inequalities 
\begin{eqnarray}
A_{k+1}&\leq &4A_k\quad\text{for }k\in\N_0, \label{eq:A4}\\
B_{k+1}&\leq &4B_k\quad\text{for }k\in\N \label{eq:B4}
\end{eqnarray}
hold. Observe also that%
\begin{align}
A_{k}  & =2A_{k+1}-3B_{k+1}\label{eq:recinvA}\\
B_{k}  & =2B_{k+1}-A_{k+1} \label{eq:recinvB}
\end{align}
for $k\in\mathbb{N}_0$. We also have the formulae%
\begin{equation}
A_{x}=\frac{1}{2}(\lambda^{x}+\lambda^{-x}),\quad B_{x}=\frac{1}{2\sqrt{3}}(\lambda
^{x}-\lambda^{-x}),\quad x\in\R \label{eq:defAB}
\end{equation}
with%
\[
\lambda=2+\sqrt{3},\quad \lambda^{-1}=2-\sqrt{3}.
\]
We remark that $\alpha=\log\lambda$.
For reference, we list the first few values of both $A_{n}$ and $B_{n}:$%
\begin{align*}
(A_{0},\ldots,A_{4}) =(1,2,7,26,97),\qquad
(B_{0},\ldots,B_{4}) =(0,1,4,15,56)
\end{align*}

\begin{lemma}
\label{lem:sumid}For $K\in\mathbb{N}_0$ we have the following formulae%
\begin{eqnarray*}
\sum_{k=0}^{K}B_{k}+B_{k+1}=A_{K+1}-1, \quad 2\sum_{k=0}^K A_k=3B_{K+1}-A_{K+1}+1, \\
\sum_{k=0}^{K}A_{k}+A_{k+1}=3B_{K+1}, \quad 2\sum_{k=0}^K B_k=A_{K+1}-B_{K+1}-1.
\end{eqnarray*}
\end{lemma}

\begin{proof}
The proof uses induction and the recurrences (\ref{eq:recA}),(\ref{eq:recB}),(\ref{eq:recinvA}) and (\ref{eq:recinvB}) for $A_n$ and $B_n$.
\end{proof}

%
%

\begin{lemma}
\label{lem:trigid}For all $n\in\mathbb{N}$ and $0\leq k\leq n$ the following
equalities hold
\begin{eqnarray*}
B_kA_{n-k}+A_kB_{n-k}=B_n, \quad B_nA_{n-k}-B_{n-k}A_n=B_k,   \\
A_kA_{n-k}+3B_{n-k}B_k=A_n, \quad A_nA_{n-k}-3B_nB_{n-k}=A_k.
\end{eqnarray*}
\end{lemma}
\begin{proof}
This is only a different formulation of (\ref{eq:hyp1}) and (\ref{eq:hyp2}).
\end{proof}

\section{Equally spaced knots on $[0,1]$}
We let $t_i=i/n$ for $0\leq i\leq n$ and view the partition of points $\pi
_{n}=\{t_{i}:i\in \{0,\ldots ,n\}\}$. Additionally we set $t_{-1}=0$ and $t_{n+1}=1$. If we define $\delta_i:=t_i-t_{i-1}$ for $0\leq i\leq n+1$, we get that $\delta_i=1/n$ for $1\leq i\leq n$ this case of equally spaced knots. Furthermore, we have for the entries $(a_{i,k})$ of the inverse of the Gram matrix $(b_{i,k})=\left<N_i,N_k\right>$ consisting of the pairwise scalar products of the piecewise linear, continuous B-spline functions corresponding to the partition $\pi_n$ (see \cite{Ciesielski1966}
or \cite{CiesielskiKamont2004})
\begin{equation}
a_{i,k}=\frac{2n}{B_{n}}(-1)^{i+k}A_{i\wedge k}A_{n-i\wedge n-k}\quad \text{%
with }0\leq i,k\leq n  \label{eq:invgram}
\end{equation}%
Observe that from formula $(\ref{eq:invgram})$ it follows that
\begin{equation}
\frac{\left\vert a_{i,k}\right\vert }{|a_{i-1,k}|}=\left\{ 
\begin{array}{ll}
\frac{A_{i}}{A_{i-1}}, & \text{for }1\leq i\leq k, \\ 
\frac{A_{n-i}}{A_{n-i+1}}, & \text{for }k<i\leq n.%
\end{array}%
\right. \label{eq:quota}
\end{equation}
Let $\mathcal{S}_n$ be the space of piecewise linear continuous functions with knots $\pi_n$. For the $L^{1}$-norm (or the $L^\infty$-norm) of the projection operator $P_{n}:L^2([0,1])\rightarrow \mathcal{S}_n$ we have the formula (see for instance 
\cite{CiesielskiKamont2004})%
\[
\left\Vert P_{n}\right\Vert _{1}=\max_{0\leq k\leq
n}\sum_{i=1}^{n}p_{i,k}\varphi \left( \frac{|a_{i,k}|}{|a_{i-1,k}|}\right)
=:\max_{0\leq k\leq n}g_{k}(n), 
\]%
where $\varphi (t)=\frac{1+t^{2}}{(1+t)^{2}}$ and $p_{i,k}=\frac{\delta
_{i}}{2}(|a_{i,k}|+|a_{i-1,k}|).$ We observe that for $t>1,$ $\varphi $ is
strictly increasing and $\varphi (t)=\varphi (1/t)$ for $t>0.$ Furthermore, $\varphi (2+\sqrt{3})=2$

Using the definition of $g_k(n)$, (\ref{eq:quota}) and the properties of $\varphi$, we obtain 
for $0\leq k\leq n$%
\begin{equation}\label{eq:gk}
g_{k}(n)=\frac{1}{B_{n}}\left(
A_{n-k}\sum_{j=0}^{k-1}(A_{j+1}+A_{j})\varphi
(A_{j+1}/A_{j})+A_{k}\sum_{j=0}^{n-k-1}(A_{j+1}+A_{j})\varphi
(A_{j+1}/A_{j})\right) . 
\end{equation}


%
%
%
%



\begin{theorem}
\label{lem:fall0}For all $n\in\mathbb{N}$, we have that $g_{0}(n+1)>g_{0}(n)$.
\end{theorem}

\begin{proof}
The expression $Dg_0(n):=g_{0}(n+1)-g_{0}(n)$ equals
\begin{equation}
\sum_{j=0}^{n}\frac{1}{B_{n+1}}(A_{j}+A_{j+1})%
\varphi \left( \frac{A_{j+1}}{A_{j}}\right) -\sum_{j=0}^{n-1}\frac{1}{B_{n}}%
(A_{j}+A_{j+1})\varphi \left( \frac{A_{j+1}}{A_{j}}\right)
\end{equation}
Thus, we get further
\begin{eqnarray*}
Dg_0(n)&=&\frac{A_{n}+A_{n+1}}{B_{n+1}}\varphi \left( \frac{A_{n+1}}{A_{n}}\right)
+\sum_{j=0}^{n-1}\left( \frac{1}{B_{n+1}}-\frac{1}{B_{n}}\right)
(A_{j}+A_{j+1})\varphi \left( \frac{A_{j+1}}{A_{j}}\right)  \\
&=&\frac{A_{n}+A_{n+1}}{B_{n+1}}\varphi \left( \frac{A_{n+1}}{A_{n}}\right)
-\left( \frac{B_{n+1}-B_{n}}{B_{n+1}B_{n}}\right)
\sum_{j=0}^{n-1}(A_{j}+A_{j+1})\varphi \left( \frac{A_{j+1}}{A_{j}}\right) 
\\
&>&\varphi \left( \frac{A_{n+1}}{A_{n}}\right) \left( \frac{A_{n}+A_{n+1}}{%
B_{n+1}}-3\frac{B_{n+1}-B_{n}}{B_{n+1}}\right)
\end{eqnarray*}%
where for the inequality, we have used Lemma \ref{lem:sumid} and the fact that
for $j\leq n-1$ we have $\varphi \left( \frac{A_{j+1}}{A_{j}}\right) <\varphi \left( \frac{A_{n+1}}{%
A_{n}}\right),
$
since $A_{j+1}/A_j<A_{n+1}/A_n$ for $j\leq n-1$. Now we use the recurrences
(\ref{eq:recA}) and (\ref{eq:recB}) to obtain that last term in the big bracket in the previous display equals zero and thus the theorem is proved.
\end{proof}

Theorem \ref{th:0grk} will then show that it suffices to have Theorem \ref%
{lem:fall0} to conclude the monotonicity of the sequence of norms of the
projection operators.

%

\begin{theorem}
\label{th:0grk}For all $n\in \mathbb{N}$ and all $k\in \mathbb{N}$ with $%
1\leq k\leq \left\lfloor n/2\right\rfloor ,$ we have%
\[
g_{0}(n)\geq g_{k}(n). 
\]
\end{theorem}

\begin{remark}
Due to symmetry, we get this inequality for all $1\leq k\leq n-1,$ and in
fact the equality $g_{0}(n)=g_{n}(n).$
\end{remark}

\begin{proof}[Proof of Theorem \ref{th:0grk}]
We first observe that for general $k$ we have to consider $h_{k}(n):=B_{n}(g_{0}(n)-g_{k}(n))$
and show that this is greater or equal zero. Now by definition
\begin{eqnarray*}
h_{k}(n) &=&\sum_{j=0}^{n-1}(A_{j}+A_{j+1})\varphi \left( \frac{A_{j+1}}{%
A_{j}}\right) -A_{n-k}\sum_{j=0}^{k-1}(A_{j}+A_{j+1})\varphi \left( \frac{%
A_{j+1}}{A_{j}}\right) \\
&&-A_{k}\sum_{j=0}^{n-k-1}(A_{j}+A_{j+1})\varphi \left( \frac{A_{j+1}}{A_{j}}%
\right) 
\end{eqnarray*}
Rearranging terms, this yields
\begin{eqnarray*}
h_k(n)&=&-(A_{k}-1)\sum_{j=0}^{n-k-1}(A_{j}+A_{j+1})\varphi \left( \frac{A_{j+1}}{%
A_{j}}\right) +\sum_{j=n-k}^{n-1}(A_{j}+A_{j+1})\varphi \left( \frac{A_{j+1}%
}{A_{j}}\right)  \\
&&-A_{n-k}\sum_{j=0}^{k-1}(A_{j}+A_{j+1})\varphi \left( \frac{A_{j+1}}{A_{j}}.
\right)
\end{eqnarray*}
Employing again the inequality $\varphi(A_{j+1}/A_j)< \varphi(A_{k+1}/A_k)$ for $j<k$, we get
\begin{eqnarray*}
h_k(n)&\geq &\varphi \left( \frac{A_{n-k+1}}{A_{n-k}}\right) \left(
-(A_{k}-1)\sum_{j=0}^{n-k-1}(A_{j}+A_{j+1})+%
\sum_{j=n-k}^{n-1}(A_{j}+A_{j+1})\right.  \\
&&\left. -A_{n-k}\sum_{j=0}^{k-1}(A_{j}+A_{j+1})\right) ,
\end{eqnarray*}%
since $k\leq \left\lfloor n/2\right\rfloor.$ Finally a calculation using Lemmas \ref{lem:sumid} and \ref{lem:trigid} shows that the big bracket equals zero and thus the theorem is proved.
\end{proof}

\begin{remark}
We observe that we can use this theorem to conclude with Theorem \ref%
{lem:fall0} that $\left\Vert P_{n}\right\Vert _{1}<\left\Vert
P_{n+1}\right\Vert _{1}.$ We also remark that we get a simple proof that $%
\left\Vert P_{n}\right\Vert _{1}<2$ for all $n\in \mathbb{N}$ and $%
\lim_{n\rightarrow \infty }\left\Vert P_{n}\right\Vert _{1}=2,$ since with
formula (\ref{eq:gk}) and the last theorem we get%
\[
\left\Vert P_{n}\right\Vert _{1}=g_{0}(n)=\frac{1}{B_{n}}%
\sum_{j=0}^{n-1}(A_{j}+A_{j+1})\varphi \left( \frac{A_{j+1}}{A_{j}}\right)
<2,
\]%
since $\varphi \left( \frac{A_{j+1}}{A_{j}}\right) <2/3$ for all $j\in \mathbb{%
N}_{0}$ and $\sum_{j=0}^{n-1}(A_{j}+A_{j+1})=3B_{n}$ by
Lemma \ref{lem:sumid}. Additionally, for $\varepsilon >0$ we choose $m$ such
that $\varphi \left( \frac{A_{m+1}}{A_{m}}\right) >(2-\varepsilon)/3 ,$ which is
possible, since $\lim_{m\rightarrow \infty }\varphi \left( \frac{A_{m+1}}{%
A_{m}}\right) 
=\varphi (2+\sqrt{3})=2/3.$ For $n>m$ we thus have%
\begin{eqnarray*}
\left\Vert P_{n}\right\Vert _{1} &=&g_{0}(n)\geq \frac{1}{B_{n}}%
\sum_{j=m}^{n-1}(A_{j}+A_{j+1})\varphi \left( \frac{A_{m+1}}{A_{m}}\right) 
\\
&>&(2-\varepsilon )\frac{B_{n}-B_{m}}{B_{n}}\rightarrow 2-\varepsilon \quad 
\text{for }n\rightarrow \infty .
\end{eqnarray*}
\end{remark}

\bibliographystyle{plain}
\bibliography{LebesgueTorus}

\end{document}